\documentclass[12pt]{amsart}

\usepackage{graphicx}
\usepackage{amsmath}
\usepackage{amscd}
\usepackage{amsfonts}
\usepackage{amssymb}
\usepackage{color}
\usepackage{fullpage}
\usepackage{mathtools}
\usepackage{graphicx}
\usepackage{caption}
\usepackage{subcaption}
\usepackage[pagebackref,hypertexnames=false, colorlinks, citecolor=red,linkcolor=blue, urlcolor=red]{hyperref}

\newcommand\red{\color{red}}

\numberwithin{equation}{section}

\newtheorem{theorem}{Theorem}[section]

\newtheorem{lemma}[theorem]{Lemma}
\newtheorem{proposition}[theorem]{Proposition}

\newtheorem{corollary}[theorem]{Corollary}

\theoremstyle{definition}
\newtheorem{definition}[theorem]{Definition}
\newtheorem{example}[theorem]{Example}
\newtheorem{remark}[theorem]{Remark}

\newtheorem{question}[theorem]{Question}

\newcommand{\be}{\begin{equation}}
\newcommand{\ee}{\end{equation}}
\newcommand{\bes}{\begin{equation*}}
\newcommand{\ees}{\end{equation*}}

\newcommand{\cH}{H}
\newcommand{\cK}{\mathcal{K}}

\newcommand{\cA}{\mathcal{A}}
\newcommand{\cB}{\mathcal{B}}

\newcommand{\cS}{\mathcal{S}}

\newcommand{\cI}{\mathcal{I}}

\newcommand{\cW}{\mathcal{W}}

\newcommand{\bC}{\mathbb{C}}
\newcommand{\bD}{\mathbb{D}}

\newcommand{\bR}{\mathbb{R}}

\newcommand{\bZ}{\mathbb{Z}}

\newcommand{\UCP}{\operatorname{UCP}}

\newcommand{\Wmin}[1]{\cW^{\text{min}}_{#1}}
\newcommand{\Wmax}[1]{\cW^{\text{max}}_{#1}}

\newcommand{\ol}{\overline}

\begin{document}

\title{Compressions of Compact Tuples}

\author{Benjamin Passer}

\address{Department of Pure Mathematics\\
University of Waterloo\\
Waterloo, ON\\
Canada}
\email{bpasser@uwaterloo.ca}

\author{Orr Moshe Shalit}
\address{Faculty of Mathematics\\
Technion - Israel Institute of Technology\\
Haifa\; 3200003\\
Israel}
\email{oshalit@technion.ac.il}

\thanks{The first author was partially supported by a 2016-2018 Zuckerman Fellowship at the Technion.
The second author was partially supported by Israel Science Foundation Grant no. 195/16.}

\subjclass[2010]{47A20, 47A13, 46L07, 47L25}
\keywords{Matrix convex set; matrix range; matrix extreme point; operator system; structure of compact tuples}
%%%%%%%%%%%%%%%%%%%%%%%%
\begin{abstract}
We study the matrix range of a tuple of compact operators on a Hilbert space and examine the notions of {\em minimal}, {\em nonsingular}, and {\em fully compressed} tuples. 
In this pursuit, we refine previous results by characterizing nonsingular compact tuples in terms of matrix extreme points of the matrix range. Further, we find that a compact tuple $A$ is fully compressed if and only if it is multiplicity-free and the Shilov ideal is trivial, which occurs if and only if $A$ is minimal and nonsingular. 
Fully compressed compact tuples are therefore uniquely determined up to unitary equivalence by their matrix ranges. We also produce a proof of this fact which does not depend on the concept of nonsingularity.
\end{abstract}

\maketitle

%%%%%%%%%%%%%%%%%%%%%%%%%
\section{Background and statement of main results}

Let $A = (A_1, \ldots, A_d) \in \cB(H)^d$ be a $d$-tuple of bounded operators on a Hilbert space $H$.
We write $S_A$ for the operator system generated by $A$, and $\UCP(S_A, M_n)$ for the space of all unital completely positive maps from $S_A$ into the algebra $M_n = M_n(\bC)$ of $n \times n$ matrices.
The {\em matrix range} of $A$ is the matrix convex set $\cW(A) = \sqcup_{n=1}^\infty \cW_n(A) \subseteq \sqcup_{n=1}^\infty M_n^d$, where
\[
\cW_n(A) := \left\{\phi(A) : \phi \in \UCP(S_A, M_n) \right\} \subseteq M_n^d,
\]
and $\phi(A) :=(\phi(A_1), \ldots, \phi(A_d))$.
The matrix range of a single operator was introduced by Arveson in \cite{Arv70,Arv72}, and the matrix range of a $d$-tuple $A$ is an organic extension of the concept. In particular, it determines the operator space $S_A$ up to unit preserving and completely isometric isomorphism (see \cite[Theorem 2.4.2]{Arv72} or \cite[Theorem 5.1]{DDSSa}).
The matrix range has been used and studied in recent works, in the contexts of the UCP interpolation problem \cite{DDSS} (following \cite{HKM13}), finite-dimensional/compact representability of operator systems \cite{Passer,PSS18} (following \cite{FNT}), and extremal problems in matrix convex sets \cite{Eve18} (following \cite{EHKM16}).

The main purpose of this note is to explore the extent to which the matrix range determines a $d$-tuple up to unitary equivalence, under suitable assumptions.
We unify the treatment in \cite{DDSSa} and \cite{Passer} by considering the problem for tuples of compact operators, and separately for tuples of normal operators. It is clear that the matrix range does not detect multiplicity (e.g., $\cW(A) = \cW(A \oplus A)$), thus one needs to impose some kind of minimality condition.

\begin{definition}
A $d$-tuple $A \in \cB(H)^d$ is said to be {\em minimal} if there is no proper closed reducing subspace $G \subset H$ such that $\cW(P_G A |_G) = \cW(A)$.
\end{definition}

This notion of minimality was used in \cite[Section 6]{DDSS}, and in the finite-dimensional case it corresponds precisely to the notion of {\em minimal  pencil} used earlier in \cite{HKM13} (and to the notion of {\em $\sigma$-minimal pencil} used in \cite{Zalar}). Using matricial polar duality \cite{EW97}, one can show that the results of \cite{HKM13} imply that a minimal $d$-tuple of operators on a finite-dimensional space is determined up to unitary equivalence by its matrix range.
In \cite[Section 6]{DDSS}, this problem was treated in the wider setting of compact tuples. It was claimed there, mistakenly, that if $A$ and $B$ are two minimal tuples of compact operators, then $\cW(A) = \cW(B)$ if and only if $A$ is unitarily equivalent to $B$. This is false, as \cite[Example 3.14]{Passer} shows.  After the mistake was discovered, a corrected version of that paper appeared on the arXiv \cite{DDSSa}. 
The correct result, \cite[Theorem 6.9]{DDSSa}, shows that if two minimal tuples of compact operators have the same matrix range and also have a new property called \textit{nonsingularity}, then they are unitarily equivalent. The property was used as a somewhat ad hoc fix, so one of our goals here is to present different conditions which allow us to avoid the nonsingularity assumption.

In order define nonsingularity, we need to review some basic facts about representations of C*-algebras of compact operators (see \cite[Section I.10]{DavBook} for proofs of the facts stated in this paragraph). Let $\cK(H)$ denote the algebra of compact operators on a Hilbert space $H$. If $\cA$ is a C*-subalgebra of $\cK(H)$, then every representation of $\cA$ is the direct sum of irreducible representations, and every nonzero irreducible representation of $\cA$ is unitarily equivalent to a direct summand of the identity representation. It follows that if $A \in \cK(H)^d$, then $C^*(A)$ (which is not assumed unital) is given as a direct sum $C^*(A) = \oplus_{i\in I} \cA_i$, where for every $i\in I$, the algebra $\cA_i$ is either unitarily equivalent to $\cK(H_i) \otimes I_{K_i}$ for some Hilbert spaces $H_i$ and $K_i$, or $\cA_i = 0$. In particular, $A$ is the direct sum of irreducible compact $d$-tuples, some of which may be zero. If there are no two irreducible summands that are unitarily equivalent, we say that $A$ is {\em multiplicity-free}.

Recall that $S_A$ denotes the operator system generated by $A$.
Thus, the C*-algebra $C^*(S_A)$ generated by $S_A$ is just the unital C*-algebra generated by $A$.
Every nondegenerate representation of $C^*(A)$ extends uniquely to a unital representation of $C^*(S_A)$.
When $\dim H < \infty$, the irreducible representations of $C^*(S_A)$ are precisely the unitizations of irreducible subrepresentations of $C^*(A)$.  
When $\dim H = \infty$, $C^*(S_A)$ may have an additional kind of representation, {\em the singular representation} $\pi_0 : C^*(S_A) \to \bC$, determined by $\pi_0(I) = 1$ and $\pi_0(A_i)= 0$ for $i \in \{1, \ldots, d\}$.
The singular representation $\pi_0$ may or may not be equivalent to a subrepresentation of the identity representation.

We shall also require the theory of boundary representations and the C*-envelope \cite{Arv69,Arv72,ArvNote}.
Recall that a {\em boundary representation} for $S_A$ in $C^*(S_A)$ is an irreducible unital representation $\pi : C^*(S_A) \to \cB(H_\pi)$, such that $\pi$ is the unique UCP extension of $\pi\big|_{S_A}$ to $C^*(S_A)$.
An ideal $J \triangleleft C^*(S_A)$ is called a {\em boundary ideal} (for $S_A$) if the quotient map $C^*(S_A) \to C^*(S_A) / J$ is completely isometric on $S_A$.
The {\em Shilov ideal} is the largest boundary ideal, and the {\em C*-envelope} of $S_A$ is the quotient of $C^*(S_A)$ by the Shilov ideal.

The C*-envelope of $S_A$ can also be identified with the image of $C^*(S_A)$ under the sum of all boundary representations (see \cite[Theorem 7.1]{ArvChoquet1} or \cite[Theorem 3.4]{DavKen15}). Let us write $\partial_{A}$ for the collection of irreducible subrepresentations of the identity representation of $C^*(S_A)$ which are also boundary representations for $S_A$ in $C^*(S_A)$.
Thus, when $H$ is infinite dimensional, we may write
\[
C^*_e(S_A) \cong \sigma(C^*(S_A)) \oplus \bigoplus_{\pi \in \partial_A} \pi(C^*(S_A)),
\]
where $\sigma$ is $\pi_0$ if $\pi_0$ is a boundary representation, and $\sigma$ is the nil representation otherwise. Some of the summands might be redundant, since $\pi_0$ might be a boundary representation as well as a subrepresentation of the identity representation. That is, it is possible that $\pi_0 \in \partial_A$.
We can now finally give the definition of nonsingularity, and the corresponding uniqueness theorem.

\begin{definition}\label{def:nonsingular_intro} \cite[Definition 6.3]{DDSSa}
A tuple $A = (A_1,...,A_d) \in \cK(H)^d$ is said to be {\em nonsingular} if either $\dim H < \infty$, or $\dim H = \infty$ and for every $n$ and every matrix $(s_{ij}) \in M_n(S_A)$,
\be\label{eq:nonsingular_intro}
\|\pi_0(s_{ij})\| \leq \sup \left\{\|\pi(s_{ij})\| : \pi \in \partial_A \right\}.
\ee
Otherwise $A$ is said to be {\em singular}.
\end{definition}

\begin{theorem}\label{thm:quotingminsingunique} \cite[Theorem 6.9]{DDSSa}
Let $A$ and $B$ be two nonsingular and minimal $d$-tuples of compact operators.
Then $\cW(A) = \cW(B)$ if and only if $A$ is unitarily equivalent to $B$.
\end{theorem}

We extend and simplify the results in \cite{DDSSa} in two separate ways. First, we identify singular and nonsingular compact tuples by studying compressions, summands, and matrix extreme points of the matrix range, as in Proposition \ref{prop:zerocompression}, Theorem \ref{thm:char_nonsing}, and Theorem \ref{thm:minsingchar}. Further, our results allow us to conclude that the examples considered by Evert in \cite{Eve18} are nonsingular. 
We also consider a different minimality condition, opting to discuss arbitrary compressions instead of compressions to reducing subspaces.

\begin{definition}\label{def:fullcom} \cite[Definition 3.20]{Passer}
A $d$-tuple $A \in \cB(H)^d$ is said to be {\em fully compressed} if there is no proper closed subspace $G \subset H$ such that $\cW(P_G A|_G) = \cW(A)$.
\end{definition}

Proposition 3.21 of \cite{Passer} gives an extremely restrictive uniqueness theorem for fully compressed compact tuples, as a result of direct computations. Namely, if $T$ is a fully compressed $d$-tuple of compact operators, and $\cW(T)$ is a matrix convex set which is generated by its first level, then $T$ is uniquely determined up to unitary equivalence. The question of whether the uniqueness result persists without the assumption about the first level was left open. In pursuit of this result, we characterize fully compressed compact tuples and nonsingular compact tuples in terms of the $C^*$-envelope, extending the results of \cite{DDSSa}. 
Our main theorem (Theorem \ref{thm:roundrobin_body}) is as follows.

\begin{theorem}\label{thm:roundrobin_intro}
Let $A \in \cK(H)^d$ be a tuple of compact operators. Then the following are equivalent.
\begin{enumerate}
\item\label{it:fullcomfirst} $A$ is fully compressed.
\item\label{it:mfShfirst} $A$ is multiplicity-free, and the Shilov ideal of $S_A$ in $C^*(S_A)$ is trivial.
\item\label{it:msfirst} $A$ is minimal and nonsingular.
\end{enumerate}
\end{theorem}

The equivalence of (\ref{it:mfShfirst}) and (\ref{it:msfirst}) is a direct improvement of \cite[Proposition 6.7]{DDSSa}, which only provided a partial version of $(\ref{it:mfShfirst}) \implies (\ref{it:msfirst})$. We also immediately obtain the following corollary of Theorems \ref{thm:quotingminsingunique} and \ref{thm:roundrobin_intro}, which appears as Corollary \ref{cor:finaluniqueness_main}.

\begin{corollary}\label{cor:uniquenessfc_intro}
Let $A$ and $B$ be two fully compressed $d$-tuples of compact operators.
Then $\cW(A) = \cW(B)$ if and only if $A$ is unitarily equivalent to $B$.
\end{corollary}

We approach the above results from two points of view. 
First, just as the introduction of nonsingularity may be used to patch the errors in \cite{DDSS}, so too may one consider fully compressed compact tuples instead of minimal ones. 
Thus, we provide proofs of the equivalence $(\ref{it:fullcomfirst}) \iff (\ref{it:mfShfirst})$ in Theorem \ref{thm:roundrobin_intro}, as well as of Corollary \ref{cor:uniquenessfc_intro}, which do not rely on the notion of nonsingularity. 
We believe this is of interest, as fully compressed tuples need not be compact (unlike nonsingular tuples), so the results are potentially open to generalization. 
However, we also find that by using the concepts of fully compressed tuples and nonsingularity in tandem, we may prove all of Theorem \ref{thm:roundrobin_intro} and unify previous results, thereby making it easier to detect tuples which meet the (equivalent) conditions. 

Finally, we close with some brief discussion of tuples which are not necessarily compact. 
In particular, in Theorem \ref{thm:cmdm_normal} we prove that a normal tuple is fully compressed if and only if it is minimal. 
Using an earlier result \cite[Theorem 3.26]{Passer}, we are thus able to give a complete description of all fully compressed normal tuples.

%%%%%%%%%%%%%%%%%%%%%%%%%%%%%%%%%%%%%%%%%%%%%
\section{Matrix convexity and extreme points}

The sets considered in the theory of matrix convexity are the ``free sets". 
For fixed $d \in \bZ^+$, we consider subsets of the form $\cS = \sqcup_{n=1}^\infty \cS_n$ contained in $\sqcup_{n=1}^\infty M_n^d$, where for every $n$ the set $\cS_n$ consists of $d$-tuples of $n \times n$ matrices. 
Below, we shall refer to $\cS_n$ as the {\em $n$th level} of $\cS$. 
A free set $\cS$ is said to be {\em matrix convex} if for every $X \in \cS_m, Y \in \cS_n$,
\[
X \oplus Y \in \cS_{m+n},
\]
and in addition, for every $\phi \in \UCP(M_m,M_k)$,
\[
\phi(X) := (\phi(X_1), \ldots, \phi(X_d)) \in \cS_k.
\]
We say that a matrix convex set $\cS$ is {\em closed}/{\em bounded} if every level $\cS_n$ is closed/bounded, and we note that if $\cS$ is bounded, there is actually a uniform norm bound that applies simultaneously to each $\cS_n$. If $T \in \cB(H)^d$ is a tuple of bounded operators, then the matrix range $\cW(T)$ is a closed and bounded matrix convex set, and in fact every closed and bounded matrix convex set arises this way \cite[Section 2.2]{DDSS}.

Matrix convexity is defined above in reference to UCP maps. From Choi's theorem (see \cite{Choi75}), a concrete version immediately follows. First, for $X = (X_1, \ldots, X_d)\in M_n^d$ and $V \in M_{n,m}$, we write
\[
V^*X V = (V^*X_1 V , \ldots, V^* X_d V) \in M_m^d.
\]
If $X^{(i)} \in \cS_{n_i}$ and $V_i \in M_{n_i,n}$ satisfy $\sum_{i=1}^k V_i^* V_i = I_n$, then the sum $\sum_{i=1}^k V_i^* X^{(i)} V_i$ is called a {\em matrix convex combination} of the $X^{(i)}$. The matrix convex combination is said to be {\em proper} if $\operatorname{rank} V_i = n_i$ for all $i$, and {\em weakly proper} if $V_i \neq 0$ for all $i$. 
A free set $\cS$ is matrix convex if and only if it is closed under matrix convex combinations.

For compact convex sets $K \subseteq \bC^d$, the Krein-Milman theorem and Milman's converse show that $K$ is the closed convex set generated by the set of extreme points, and that the set of extreme points is minimal with respect to this property. An analogous study of extreme points for matrix convex sets is more complicated, as there are multiple relevant notions of extreme point to consider.

\begin{definition}
Let $\cS$ be a matrix convex set.
A point $X \in \cS$, say $X \in \cS_n$, is said to be
\begin{enumerate}
\item a {\em Euclidean extreme point of $\cS$} if $X = tY + (1-t)Z$ with $t \in (0,1)$, $Y,Z \in \cS_n$ implies $X = Y = Z$;
\item a {\em matrix extreme point (MEP) of $\cS$} if whenever $X$ is written as a proper matrix convex combination $X = \sum_{i=1}^k V_i^* X^{(i)} V_i$, then $X^{(i)}$ is unitarily equivalent to $X$ for all $i$;
\item an {\em absolute extreme point (AEP) of $\cS$} if whenever $X$ is written as a weakly proper matrix convex combination $X = \sum_{i=1}^k V_i^* X^{(i)} V_i$, then for all $i$, the tuple $X^{(i)}$ is unitarily equivalent to $X$ or to a direct sum $X \oplus Z_i$ for some $Z_i \in \cS$.
\end{enumerate}
\end{definition}

For any convex set $K$, we use $\text{ext}(K)$ to denote the set of extreme points of $K$. In particular, if $\cS$ is matrix convex, then $\text{ext}(\cS_n)$ consists of the Euclidean extreme points of $\cS$ which lie in level $n$. We will also let $\text{MEP}(\cS)$ and $\text{AEP}(\cS)$, respectively, denote the set of matrix extreme and absolute extreme points of a matrix convex set $\cS$.

In the first level $\cS_1$ of a matrix convex set, there is no distinction between Euclidean extreme points and matrix extreme points: $\text{MEP}(\cS) \cap \cS_1 = \text{ext}(\cS_1)$. Further, Webster and Winkler proved a matricial Krein-Milman theorem \cite[Theorem 4.3]{WW99}, which says that if $\cS$ is a closed and bounded matrix convex set, then the closed matrix convex hull of $\text{MEP}(\cS)$ is $\cS$. 
However, the set of matrix extreme points is not necessarily minimal. 
For closed and bounded real free spectrahedra (that is, matrix convex sets defined by a linear inequality), absolute extreme points are a minimal spanning set \cite[Theorem 1.1]{EH18}. 
On the other hand, there are closed and bounded matrix convex sets which have no absolute extreme points at all \cite[Theorem 1.2]{Eve18}.

Let $K \subset \bC^d$ be compact and convex. If $\cS$ is a matrix convex set with $\cS_1 = K$, then $\mathcal{S}$ sits between two extremal sets, which we denote as in \cite{DDSS}. First,
\be\label{eq:Wmindefintro}
\Wmin{}(K) := \{ M \in \sqcup_{n=1}^\infty M_n^d: M \text{ has a normal dilation with joint spectrum in } K\}
\ee
is the matrix convex hull of the compact convex set $K$. We remind the reader that a tuple $T \in \cB(H)^d$ is called \textit{normal} if $T$ consists of commuting normal operators, and that \cite[Corollary 4.4]{DDSS} shows that $\Wmin{}(K)$ is the matrix range of any normal tuple $N$ whose joint spectrum satisfies $\text{conv}(\sigma(N)) = K$. Second,
\be
\Wmax{}(K) := \{ M \in \sqcup_{n=1}^\infty M_n^d: \cW_1(M) \subseteq K\}
\ee
is the largest matrix convex set whose first level is $K$. These two sets are equal precisely when $K$ is a simplex by \cite[Theorem 4.1]{PSS18} (see also \cite[Theorem 4.7]{FNT} for a similar result with the assumption that $K$ is a polytope).

In \cite{Kriel}, Kriel considers the extremal matrix convex sets whose level $n$ is specified, where $n$ is any fixed positive integer. In particular, \cite[Corollary 6.12]{Kriel} implies that if a closed and bounded matrix convex set $\cS$ of self-adjoints is equal to the matrix convex hull of $\cS_n$ for some $n$, then the absolute extreme points of $\cS$ are a minimal spanning set. 
A crucial aspect of the proof is the fact that matrix extreme points are either absolute extreme points or admit nontrivial matrix extreme dilations.

\begin{theorem}\cite[Lemma 6.11]{Kriel} \label{thm:Krielboundedlevel}
Let $\mathcal{S} \subset \sqcup_{n=1}^\infty (M_n^d)_{sa}$ be a closed and bounded matrix convex set of self-adjoints, and let $X \in \cS_m$ be a matrix extreme point of $\cS$. 
Then either $X$ is an absolute extreme point of $\cS$, or there is a matrix extreme point of $\cS$ which is of the form $\begin{pmatrix} X & b \\ b^* & c \end{pmatrix}$ for some $b \in (\bC^m)^d \setminus \{0\}$ and $c \in (\bR^m)^d$.
\end{theorem} 

While the above result is stated for self-adjoints, we may easily obtain a corresponding result in the general case by breaking a tuple into real and imaginary parts. Alternatively, the result can also be obtained by a combination of \cite[Lemma 2.3]{DavKen15} and \cite[Theorem B]{Far00}.

Below we show that under certain geometric conditions, Euclidean extreme points can automatically be absolute extreme points. Recall that from the classical Krein-Milman theorem and Milman's converse, it follows that if $K \subset \bC^d$ is a compact convex set, then any point $\lambda \in K$ satisfies
\be\label{eq:isolationEEP}
\lambda \text{ is an isolated extreme point of } K \iff \lambda \not\in \overline{\text{conv}( \text{ext}(K) \setminus \{\lambda\})}.
\ee
We write
\bes
\cI_K := \{\lambda \in \text{ext}(K): \lambda \text{ is isolated in } \text{ext}(K) \}
\ees
and note that any point $\lambda \in \cI_K$ is the vertex of some polytope which contains $K$.

\begin{proposition}\label{prop:isolatedEEPAEP}
Suppose $\cS$ is a closed and bounded matrix convex set with $\cS_1 = K$. If there is a polytope $P$ such that $P$ contains $K$ and $\lambda \in K$ is a vertex of $P$, then $\lambda$ is an absolute extreme point of $\cS$. Consequently, if $\lambda$ is an isolated extreme point of $K$, then $\lambda$ is an absolute extreme point of $\mathcal{S}$.
\end{proposition}
\begin{proof}
If $\lambda \in K$ is a vertex of the polytope $P$, then $\lambda$ is also a vertex of a simplex $\Delta$ which contains $P$. It follows that $\lambda$ is an absolute extreme point of $\Wmin{}(\Delta)$ (see \cite[Lemma 3.9]{Passer} and the commentary immediately thereafter). Since
\bes
\mathcal{S} \subseteq \Wmax{}(K) \subseteq \Wmax{}(\Delta) =\Wmin{}(\Delta),
\ees
it follows that $\lambda \in \mathcal{S}_1$ is an absolute extreme point of a set larger than $\cS$, so $\lambda$ is also an absolute extreme point of $\cS$.
\end{proof}

Motivated by the equivalence (\ref{eq:isolationEEP}), we define a collection of matrix extreme points which behave in a similar way.

\begin{definition}
Let $\mathcal{S}$ be a closed and bounded matrix convex set, and let $X \in \mathcal{S}$. Then we call $X$ a \textit{crucial matrix extreme point} if the collection
\bes
\mathcal{C} := \{ M \in \mathcal{S}:   M \text{ is a matrix extreme point of } \mathcal{S} \text{ and } M \text{ is not unitarily equivalent to } X\}
\ees
has the property that the closed matrix convex hull of $\mathcal{C}$ excludes $X$.
\end{definition}

Note that by Webster and Winkler's matricial Krein-Milman theorem, a crucial matrix extreme point of $\cS$ is indeed a matrix extreme point of $\cS$. We also immediately reach the following from (\ref{eq:isolationEEP}):
\bes
\cS_1 \cap \{X \in \cS: X \text{ is a crucial matrix extreme point of } \cS\} \subseteq \cI_{\cS_1}.
\ees
That is, a crucial matrix extreme point of $\cS$ which belongs to the first level $\mathcal{S}_1$ must be an isolated extreme point of $\cS_1$. For sets of the form $\Wmin{}(K)$, the converse also holds.

\begin{proposition}\label{prop:minminminmin}
Let $K \subset \bC^d$ be compact and convex. Then
\be\label{eq:ittakesallkinds}
\emph{ext}(K) = \emph{MEP}(\Wmin{}(K)) = \emph{AEP}(\Wmin{}(K))
\ee
and
\be
\cI_K = \{X \in \Wmin{}(K): X \emph{ is a crucial matrix extreme point of } \Wmin{}(K)\}.
\ee
\end{proposition}
\begin{proof}
For any $X \in \Wmin{}(K)$, $X$ admits a normal matrix dilation $N \in M_m^d$ with $\sigma(N) \subseteq K$ by (\ref{eq:Wmindefintro}) and \cite[Theorem 7.1]{DDSS}. The joint diagonalization of $N$ shows that $X$ can be written as a proper matrix convex combination of points $\lambda_i$ in $K = \mathcal{S}_1$. If $X$ is a matrix extreme point of $\Wmin{}(K)$, then all the $\lambda_i$ are unitarily equivalent to $X$, so $X$ is in level one. Since $X$ is certainly still extreme, we conclude that $\text{MEP}(\Wmin{}(K)) \subseteq \text{ext}(K)$. From \cite[Lemma 3.9]{Passer}, we have that $\text{ext}(K) \subseteq \text{AEP}(\Wmin{}(K))$, and finally the containment $\text{AEP}(\Wmin{}(K)) \subseteq \text{MEP}(\Wmin{}(K))$ is trivial.

If $X$ is an isolated extreme point of $K$, then the set of matrix extreme points which are not unitarily equivalent to $X$ is precisely $\mathcal{C} := \text{ext}(K) \setminus \{X\}$ by (\ref{eq:ittakesallkinds}). The closed matrix convex set generated by $\mathcal{C}$ is the matrix range of $N := \bigoplus\limits_{\lambda \in \mathcal{C}} \lambda$, which by \cite[Corollary 4.4]{DDSS} is $\Wmin{}(\overline{\text{conv}(\mathcal{C})})$.  Since the first level $\overline{\text{conv}(\mathcal{C})}$ excludes $X$ by (\ref{eq:isolationEEP}), we see that $X$ is a crucial matrix extreme point of $\Wmin{}(K)$.

Suppose instead that $X$ is a crucial matrix extreme point of $\Wmin{}(K)$. Since $X$ is a matrix extreme point, (\ref{eq:ittakesallkinds}) shows that $X \in \text{ext}(K)$, and the fact that $X$ is crucial implies that $X \not\in \ol{\text{conv}(\text{ext}(K) \setminus \{X\})}$. By (\ref{eq:isolationEEP}), $X$ is isolated as an extreme point of $K$.
\end{proof}

An immediate consequence of Proposition \ref{prop:minminminmin} is that a closed and bounded matrix convex set might have no crucial matrix extreme points. For example, consider $\Wmin{}(K)$ where $K$ is the unit disk. The following result also shows that in general, crucial matrix extreme points must also be absolute extreme points.

\begin{proposition}
Let $\mathcal{S}$ be a closed and bounded matrix convex set. Then every crucial matrix extreme point of $\mathcal{S}$ is an absolute extreme point.
\end{proposition}
\begin{proof}
Suppose $X$ is a matrix extreme point of $\mathcal{S}$ which is not an absolute extreme point. Theorem \ref{thm:Krielboundedlevel} shows that there is a nontrivial dilation $Y$ of $X$ which is also a matrix extreme point of $\mathcal{S}$. Moreover, since $X$ and $Y$ have distinct finite dimension, $Y$ cannot be unitarily equivalent to $X$. We conclude that since $X$ is in the matrix convex hull of $Y$, $X$ cannot be a crucial matrix extreme point of $\mathcal{S}$ by definition.
\end{proof}

In the next section, we extend the results of \cite{DDSSa} by showing how the crucial matrix extreme points of the matrix range may be used to characterize when a tuple of compact operators is nonsingular.

%%%%%%%%%%%%%%%%%%%%%%%%%%%%%%%%%%%%%%%%%%%%%
\section{Characterizations of Nonsingularity}

Let us recall some definitions surrounding nonsingular compact tuples. If $A \in \cK(H)^d$ is a tuple of compact operators on an infinite-dimensional space, then the operator system $S_A$ is not contained in the compacts, as it by definition includes the unit. Indeed, the unital $C^*$-algebra generated by $A$, denoted $C^*(S_A)$, admits a singular representation $\pi_0$ which annihilates all compact operators and maps the identity operator to $1 \in \bC$. It turns out that $\pi_0$ may or may not be dominated by $\partial_A$, the collection of irreducible subrepresentations of the identity representation of $C^*(S_A)$ which are also boundary representations of $S_A$ in $C^*(S_A)$.

\begin{definition}\label{def:nonsingular} \cite[Definition 6.3]{DDSSa}
A tuple $A = (A_1,...,A_d) \in \cK(H)^d$ is said to be {\em nonsingular} if either $\dim H < \infty$, or if $\dim H = \infty$ and for every $n$ and every matrix $(s_{ij}) \in M_n(S_A)$,
\be\label{eq:nonsingular}
\|\pi_0(s_{ij})\| \leq \sup \left\{\|\pi(s_{ij})\| : \pi \in \partial_A \right\}.
\ee
Otherwise $A$ is said to be {\em singular}.
\end{definition}

The notion of nonsingularity behaves very well with respect to direct summands and multiplicity.

\begin{lemma}\label{lem:mosummandslessproblems}
Let $A$ and $B$ be $d$-tuples of compact operators, and assume that $B$ is a summand of $A$ with $\cW(B) = \cW(A)$. If $B$ is nonsingular, then $A$ is nonsingular. Similarly, if $C \in \cK(H)^d$ and $\bigoplus\limits_{i=1}^N C$ is nonsingular for some $N \in \bZ^+$, then $C$ is nonsingular.
\end{lemma}
\begin{proof}
The equality $\cW(A) = \cW(B)$ implies that the map $A_i \mapsto B_i$ extends to a completely isometric isomorphism from $S_A$ to $S_B$ \cite[Theorem 5.1]{DDSS}, so it extends to a $*$-isomorphism $C^*_e(S_A) \to C^*_e(S_B)$.
Since $B$ is nonsingular,
\[
C_e^*(S_A) \cong C_e^*(S_B) \cong \bigoplus_{\pi \in \partial_B} \pi\left(C^*(S_B)\right).
\]
Now, the compression of $A$ to $B$ extends to a $*$-homomorphism, so we can identify $\partial_B$ with a subset of $\partial_A$ (see \cite[Theorem 2.1.2]{Arv69}).
It follows that for {\em any} representation $\sigma$ of $C^*(S_A)$ and for any $(s_{ij}) \in M_n(S_A)$,
\[
\|\sigma(s_{ij})\| \leq \sup_{\pi \in \partial_B} \|\pi(s_{ij})\| \leq \sup_{\pi \in \partial_A} \|\pi(s_{ij})\| .
\]
This shows that $A$ must be nonsingular.

Next, let $C \in \cK(H)^d$ and define $D = \bigoplus\limits_{i=1}^N C$, so the map $x \mapsto x \otimes I_N$ is a $*$-isomorphism between $C^*(S_C)$ and $C^*(S_D)$.
The unitary equivalence classes of irreducible subrepresentations of the identity representations of $C$ and $D$ are the same, hence
\[
\sup_{\pi \in \partial_C} \|\pi(s_{ij})\| = \sup_{\sigma \in \partial_D} \|\sigma(s_{ij} \otimes I_N)\|
\]
for every $(s_{ij}) \in M_n(S_C)$.
If $\dim H = \infty$, then the corresponding singular representations satisfy $\pi_0(s_{ij}) = \pi_0(s_{ij} \otimes I_N)$, and we conclude that $C$ is singular if and only if $D$ is.
\end{proof}

Nonsingularity may be detected using any of the following conditions.

\begin{proposition}\cite[Proposition 6.6]{DDSSa}\label{prop:DDSSlist}
The following conditions are sufficient for a tuple $A \in \cK(H)^d$ to be nonsingular:
\begin{enumerate}
\item $\dim H < \infty$, or
\item\label{it:0summand_} $A$ contains $0$ as a direct summand, or
\item $0$ is not an isolated extreme point of $\cW_1(A)$.
\end{enumerate}
\end{proposition}

We will strengthen sufficient condition (\ref{it:0summand_}), after which we will present a separate theorem which characterizes nonsingularity completely. To accomplish the first goal, we use a notion crucial to the arguments in \cite{Eve18}.

\begin{proposition}\label{prop:zerocompression}
Let $A \in \cK(H)^d$ be a tuple of compact operators. If $0$ is a compression of $\bigoplus\limits_{1=1}^N A$ for some $N \in \bZ^+$, then $A$ is nonsingular.
\end{proposition}
\begin{proof}
By Lemma \ref{lem:mosummandslessproblems}, we need only prove nonsingularity of $\widetilde{A} := \bigoplus\limits_{i=1}^N A$.
Let $0$ be a compression of $\widetilde{A}$.
If $0$ is a direct summand of $\widetilde{A}$, then $\widetilde{A}$ is nonsingular by Proposition \ref{prop:DDSSlist}, so we may assume otherwise.
In particular, we have that $\widetilde{A}$ is a nontrivial dilation of $0$.
It follows that there is some two-dimensional compression $B$ of $\widetilde{A}$ such that $B$ is a nontrivial dilation of $0$.
From \cite[Lemma 3.5]{EHKM16}, we conclude that $0$ is not an absolute extreme point of $\cW(\widetilde{A})$.
Finally, Proposition \ref{prop:isolatedEEPAEP} shows that $0$ is not an isolated extreme point of $\cW_1(\widetilde{A})$, and nonsingularity of $\widetilde{A}$ follows from Proposition \ref{prop:DDSSlist}.
\end{proof}
\begin{remark} Said differently, any tuple $A \in \cK(H)^d$ which satisfies the condition \lq\lq$0$ is in the finite interior of the noncommutative convex hull\rq\rq\hspace{0pt} from \cite[\S 1.2]{Eve18} is automatically nonsingular.\end{remark}

Next, we show that the conditions of Proposition \ref{prop:DDSSlist} may be adapted to form a characterization of nonsingularity, primarily by modifying the proof to consider extreme points in various levels. We will need the following observation: given a compact tuple $A \in \cK(H)^d$, the set
\be\label{eq:nonzeroMEPstuff}
\mathcal{C} := \text{MEP}(\cW(A)) \setminus \{0\}
\ee
has the property that any $X \in \mathcal{C}$ is the compression of a summand of $A$. Indeed, if $X \in \mathcal{C}$ is of size $n \times n$, then \cite[Theorem B]{Far00} shows that there is a {\em pure} UCP map $\phi \in \UCP(S_A, M_n)$ with $\phi(A) = X$. 
By \cite[Theorem 2.4]{DavKen15}, $\phi$ is the compression of the restriction of a boundary representation $\rho$.
Since $X \not= 0$, we must have $\rho(A) \not= 0$, and hence $\rho$ is also (up to unitary equivalence) a subrepresentation of the identity representation -- that is, $\rho \in \partial_A$. In other words, $\rho(A)$ is an irreducible summand of $A$, and $X$ is a compression of this summand.
Further, if $X \in \mathcal{C}$ is actually an absolute extreme point, then we must have $X = \rho(A)$ (see also \cite[Corollary 6.27]{Kriel}).

On the other hand, if $A$ is infinite-dimensional, then the point $0 \in \cW(A)$ corresponds to the UCP map sending $A_i \mapsto 0$, which is the restriction of the singular representation $\pi_0$ to $S_A$. Even if $\pi_0$ is a boundary representation, it might not be a subrepresentation of the identity. Regardless, we may guarantee that $A$ is nonsingular if sufficiently many nonzero matrix extreme points exist.

\begin{theorem}\label{thm:char_nonsing}
Let $A \in \cK(H)^d$ be a tuple of compact operators. Then $A$ is nonsingular if and only if one of the following conditions holds.
\begin{enumerate}
\item\label{it:finsum} There is a finite-dimensional summand $B$ of $A$ with $\mathcal{W}(A) = \mathcal{W}(B)$, or
\item\label{it:noncruc} $0$ is not a crucial matrix extreme point of $\mathcal{W}(A)$.
\end{enumerate}
\end{theorem}
\begin{proof}
If $B$ is a finite-dimensional summand of $A$ (perhaps equal to $A$), then certainly $B$ is nonsingular. From Lemma \ref{lem:mosummandslessproblems}, it follows that if $\cW(B) = \cW(A)$, then $A$ is also nonsingular.

Next, suppose that $A$ is infinite-dimensional and $0$ is not a crucial matrix extreme point of $\cW(A)$, so the collection $\mathcal{C}$ of (\ref{eq:nonzeroMEPstuff}) has the property that $0$ can be approximated by matrix convex combinations of points in $\mathcal{C}$. Since any point $X \in \mathcal{C}$ dilates to $\rho(A)$ for some $\rho \in \partial_A$, it follows that $\pi_0$ may be dominated on $M_n(S_A)$ by the collection $\partial_A$.
That is, the inequality \eqref{eq:nonsingular} in Definition \ref{def:nonsingular} holds, and $A$ is nonsingular.

To prove the converse, suppose that $A$ is nonsingular, and consider nonzero matrix extreme points $Y \in \mathcal{C}$. Following the logic in the proof of \cite[Corollary 6.12]{Kriel}, use Theorem \ref{thm:Krielboundedlevel} to produce successive nontrivial dilations $Y^{(1)} \prec Y^{(2)} \prec \ldots $ of nonzero matrix extreme points. Such a sequence will terminate if and only if some $Y^{(i)}$ is an absolute extreme point of $\cW(A)$, so we consider cases.

\vspace{.1 in}

\noindent \textbf{Case I.} Suppose that there is an infinite sequence $Y^{(1)} \prec Y^{(2)} \prec \ldots $ of nonzero matrix extreme points. 
We may then form an orthonormal sequence $\{e_i\}_{i=1}^\infty$ in $H$ such that $e_i^* A e_i = e_i^* Y^{(i)} e_i$ for each $i$. 
Since $A$ is compact, it follows that $0$ is in the closed matrix convex hull of $\mathcal{C}$. 
That is, $0$ is not a crucial matrix extreme point of $\cW(A)$, and condition (\ref{it:noncruc}) holds.

\vspace{.1 in}

\noindent \textbf{Case II.} Suppose that any nonzero matrix extreme point of $\cW(A)$ may be dilated to an absolute extreme point.

If there are finitely many (non unitarily-equivalent) nonzero absolute extreme points $X^{(1)}, \ldots, X^{(N)}$ of $\cW(A)$, then one of the finite-dimensional tuples $0 \oplus \bigoplus\limits_{i=1}^N X^{(i)}$ or $\bigoplus\limits_{i=1}^N X^{(i)}$ has the same matrix range as $A$. It follows that there is a finite-dimensional (hence nonsingular) minimal tuple $Y$ with $\cW(Y) = \cW(A)$. On the other hand, \cite[Corollary 6.8]{DDSSa} shows there is a summand $B$ of $A$ which is minimal and has $\cW(B) = \cW(A)$. Further, the proof of the result produces a choice of $B$ which is itself nonsingular. From Theorem \ref{thm:quotingminsingunique}, $B$ is unitarily equivalent to $Y$, and condition (\ref{it:finsum}) holds.

If there are infinitely many (non unitarily-equivalent) nonzero absolute extreme points $X^{(1)}, X^{(2)}, \ldots$ of $\cW(A)$, then since each $X^{(i)}$ is unitarily equivalent to a distinct summand of the compact tuple $A$, it follows that $0$ is in the closed matrix hull of the $X^{(i)}$. As such, $0$ is not a crucial matrix extreme point of $\cW(A)$, and condition (\ref{it:noncruc}) holds.
\end{proof}

Of course, since nonsingular compact tuples have now been characterized, we immediately find that singular compact tuples are characterized using the negation. From revisiting the proof, one can also see that \textit{minimal} singular compact tuples must be essentially of the form outlined in \cite[Example 3.14 and Corollary 3.15]{Passer}.

\begin{theorem}\label{thm:minsingchar}
Let $A \in \cK(H)^d$ be a {\red \textbf{singular}} tuple of compact operators. Then there is an integer $N \geq 1$ and a decomposition
\bes
A \cong \bigoplus_{i=1}^N X^{(i)} \oplus Y,
\ees
where $X^{(1)}, \ldots, X^{(N)}$ are (up to unitary equivalence) all the nonzero absolute extreme points of $\mathcal{W}(A)$, $Y$ is an infinite-dimensional compact tuple, and $\cW(A) = \cW\left(\bigoplus\limits_{i=1}^N X_i \oplus 0 \right)$. Moreover, if $A$ is minimal, then $Y$ is irreducible.
\end{theorem}
\begin{proof}
Suppose $A$ is singular, so certainly $A \not= 0$ and $\cW(A)$ has at least one nonzero matrix extreme point. From Theorem \ref{thm:char_nonsing}, we have that $0$ is a crucial matrix extreme point of $\cW(A)$. In particular, $0$ is an isolated extreme point of $\cW_1(A)$ and hence an absolute extreme point of $\cW(A)$ by Proposition \ref{prop:isolatedEEPAEP}. 

Using the same arguments as in the proof of Theorem \ref{thm:char_nonsing}, we conclude that since $0$ is a crucial matrix extreme point, it holds that any matrix extreme point $X \not= 0$ is a compression of an absolute extreme point, and there can be only finitely many nonzero absolute extreme points $X^{(1)}, \ldots, X^{(N)}$. In particular, we note that $N \not= 0$: $\cW(A)$ has at least one nonzero matrix extreme point, which dilates to an absolute extreme point. 

Since each $X^{(i)}$ must be unitarily equivalent to a summand of $A$, we may write 
\bes
A \cong \bigoplus_{i=1}^N X^{(i)} \oplus Y
\ees
for some compact tuple $Y$. However, $A$ is necessarily infinite-dimensional, so we must have that $Y$ is infinite-dimensional. It also holds that $\cW\left(\bigoplus\limits_{i=1}^N X^{(i)} \oplus 0\right) = \cW(A)$, as the left hand side includes every matrix extreme point of $\cW(A)$. Finally, if $A$ is minimal, then $Y$ must be irreducible, as otherwise $Y$ has an infinite-dimensional summand which detects $0$.
\end{proof}

In the next section, we examine the assumption that a compact tuple is fully compressed, aiming to prove a uniqueness theorem for the matrix range and relate the assumption to minimality and nonsingularity.

%%%%%%%%%%%%%%%%%%%%%%%%%%%%%%%%%%%%%%%%%%%%%
\section{Compressions}\label{sec:final}

Recall that a compact tuple $A \in \cK(H)^d$ is called {\em fully compressed} if no compression of $A$ to a proper closed subspace has the same matrix range as $A$. A tuple which is fully compressed is automatically minimal, but the converse need not hold, as most proper subspaces of $\cH$ are not reducing subspaces of $A$. In what follows, we consider the following questions.

\begin{itemize}
\item If $A \in \cK(H)^d$ is fully compressed, does $\cW(A)$ uniquely determine $A$?
\item How does the assumption that $A$ is fully compressed relate to previous conditions, like minimality and nonsingularity?
\end{itemize}

Nonsingularity (Definition \ref{def:nonsingular_intro}) is an assumption added to the correction \cite{DDSSa} of \cite{DDSS}, after counterexamples were found to the theorems therein. Particularly, a compact tuple which is minimal \textit{and nonsingular} is uniquely determined by its matrix range. This assumption is useful in that it directly solves the problem at hand, but it is somewhat ad-hoc. 
Therefore, our approach in this section is two-pronged.

First, we wish to show that \lq\lq fully compressed\rq\rq\hspace{0pt} is a suitable replacement condition for \lq\lq minimal and nonsingular\rq\rq\hspace{0pt}, in that a uniqueness theorem follows by appropriately adapting the techniques of \cite{DDSS} in a natural way. 
That is, we may prove a uniqueness theorem for compact tuples by examining arbitrary compressions \textit{instead of} nonsingular tuples. We choose this approach to begin with, primarily because nonsingularity is defined only in reference to compact tuples, whereas any tuple of operators may be fully compressed. Thus, this approach is open to potential generalization.

Second, we see how the study of fully compressed tuples can help us better understand the condition of nonsingularity. For example, \cite[Proposition 6.7]{DDSSa} claims that a compact tuple $A$ which is minimal and nonsingular must also have the following properties:
\begin{itemize}
\item $A$ is multiplicity-free
\item The Shilov ideal of $S_A$ in $C^*(S_A)$ is trivial.
\end{itemize}
However, only a partial converse is given. We will complete the converse in a round-robin proof, ultimately finding that for compact tuples, \lq\lq fully compressed\rq\rq\hspace{0pt} means the same thing as \lq\lq minimal and nonsingular\rq\rq\hspace{0pt}. Therefore, examining arbitrary compressions \textit{in addition to} nonsingularity leads to a somewhat better understanding of both conditions in the compact setting.

We begin with a lemma considering finite-dimensional tuples.

\begin{lemma}\label{lem:findimminfc}
If $A \in \cK(H)^d$ is minimal and $0$ is a summand of $A$, then $H$ is finite-dimensional. Moreover, if $B$ is a $d$-tuple of operators on a finite-dimensional space, then $B$ is minimal if and only if it is fully compressed.
\end{lemma}
\begin{proof}
Suppose $A \in \cK(H)^d$ is such that $A = B \oplus 0$ and $A$ is infinite-dimensional. Then $B$ is infinite-dimensional, $0 \in \cW(B)$, and $\cW(A) = \cW(B)$. We conclude that $A$ is not minimal.

For the second statement, we need only prove that minimal tuples of finite-dimensional operators are fully compressed. 
If $B$ is minimal, then because $B$ acts on a space of finite dimension, there is a compression $C$ of $B$ such that $\cW(B) = \cW(C)$ and $C$ is fully compressed. 
Both $B$ and $C$ are minimal, so by Theorem \ref{thm:quotingminsingunique}, they are unitarily equivalent.
It follows that $C$ cannot be a proper compression, and hence $B$ is fully compressed.
\end{proof}

Note that while Theorem \ref{thm:quotingminsingunique} (i.e., \cite[Theorem 6.9]{DDSSa}) concerns nonsingular tuples, we have only used it for tuples of finite-dimensional operators in the above proof. The uniqueness result we need may therefore ultimately be deduced from results in free spectrahedra, as studied in \cite{HKM13} and \cite{Zalar}. Namely, so long as one applies a shift to ensure $0$ is present in the matrix range $\cW(A)$ of a matrix tuple $A \in M_n^d$, it follows that $\cW(A)$ is the polar dual of the free spectrahedron determined by $A$ (as in \cite[Proposition 3.3 and Lemma 3.4]{DDSSa}).

\begin{lemma}\label{lem:OGequivalence_fc_redundant}
If $A \in \cK(H)^d$ is a tuple of compact operators, then $A$ is fully compressed if and only if $A$ is multiplicity-free and the Shilov ideal of $S_A$ in $C^*(S_A)$ is trivial.
\end{lemma}
\begin{proof}
If $A$ is fully compressed, then it is clearly multiplicity-free. Therefore, we may find a decomposition $H = \bigoplus\limits_{i \in I} H_i$ with $C^*(S_A) =  \bC I_H +  \bigoplus\limits_{i \in I} \cK_i$, where $\cK_i = \cK(H_i)$ is the space of compact operators on $H_i$. If $H$ is finite-dimensional, the presence of the unit is redundant. The Shilov ideal $J \triangleleft C^*(S_A)$ has the form $\bigoplus\limits_{i \in I_0} \cK_i$ for some $I_0 \subseteq I$, so we define $I_1 = I \setminus I_0$ and $G_k = \bigoplus\limits_{i \in I_k} H_i$ for $k \in \{0, 1\}$.

If $G_0$ is finite-dimensional, or if $G_0$ and $G_1$ are both infinite-dimensional, then the quotient $C^*_e(S_A) = C^*(S_A)/J$ is $*$-isomorphic to $\bC I_{G_1} + \bigoplus\limits_{i \in I_1} \cK_i$, where the sum is taken in $\cB(G_1)$. It follows that the compression onto $G_1$ is completely isometric on $S_A$, and hence $\cW(A) = \cW(P_{G_1} A|_{G_1})$. Since $A$ is fully compressed, we conclude that $G_1 = H$, and $J$ is trivial.

Suppose instead that $G_0$ is infinite-dimensional and $G_1$ is finite-dimensional. In this case,
\begin{equation}\label{eq:thelastenvelope}
C^*_e(S_A)  \,\, = \,\, C^*(S_A) /J \,\, \cong \,\, \bC I_H  + \bigoplus_{i \in I_1} \cK_i  \,\, \cong \,\, \bC \oplus \bigoplus_{i \in I_1} \cK_i,
\end{equation}
where the first sum is in $\cB(H)$ and the second sum is an external direct sum. Define $X = P_{G_1} A|_{G_1}$, so that the image of $A$ through (\ref{eq:thelastenvelope}) is $0 \oplus X$, and hence $\cW(A) = \cW(0 \oplus X)$. If $F$ is an infinite-dimensional proper subspace of $G_0$, and $B$ is the compression of $A$ to the subspace $F \oplus G_1 \subset H$, then there is a UCP map sending $B \mapsto 0 \oplus X$. We conclude that
\[
\cW(0 \oplus X) \subseteq \cW(B) \subseteq \cW(A) = \cW(0 \oplus X),
\]
and hence $\cW(B) = \cW(A)$. This contradicts the assumption that $A$ is fully compressed, so this case does not occur.

Next, we consider the converse. If $A$ is multiplicity-free and the Shilov ideal is trivial, then \cite[Proposition 6.7]{DDSSa} implies that $A$ is minimal. (Note that the portion of \cite[Proposition 6.7]{DDSSa} we are actually using is a carry-over from \cite{DDSS} and does not require nonsingularity). From Lemma \ref{lem:findimminfc}, it follows that if $A$ acts on a finite-dimensional space, which must happen if $0$ is a summand of $A$, then $A$ is fully compressed.

Assume that $A$ acts on an infinite-dimensional space, so $0$ is not a summand of $A$. Moreover, note that triviality of the Shilov ideal directly implies that $C^*(S_A) = C^*_e(S_A)$. 
Suppose that $B$ is a compression of $A$ to a closed subspace $G \subseteq H$, with $\cW(A) = \cW(B)$. Since $S_A$ and $S_B$ are then completely isometrically isomorphic, the universal property of $C^*_e(S_A)$ gives rise to a surjective $*$-homomorphism $\pi: C^*(S_B) \to C^*(S_A)$ that extends the complete isometry $B \mapsto A$. 
By the representation theory of algebras of compact operators, and keeping in mind that $A$ does not have $0$ as a direct summand, the map $\pi$ is a direct sum of representations unitarily equivalent to subrepresentations of the identity representation of $C^*(S_B)$. 
Since $A$ is minimal, every equivalence class of a subrepresentation of the identity representation of $C^*(S_B)$ appears at most once. 
We therefore find that $A$ is unitarily equivalent to a compression of $B$ to a reducing subspace $F \subseteq G$. Since the compression of $A$ to $F$ is a unitary equivalence, it extends to a $*$-representation of $C^*(S_A)$. 
By Sarason's lemma \cite[Lemma 0]{Sarason}, $F$ is semi-invariant for $C^*(S_A)$. 
But a semi-invariant subspace for a C*-algebra is a reducing subspace, so the compression of $A$ to $F$ is actually a direct summand with the same matrix range. Since $A$ is minimal, we conclude that $F = H$, and hence $G = H$. 
That is, $A$ is fully compressed.
\end{proof}

We may now show that fully compressed compact tuples are uniquely determined by their matrix ranges, without using nonsingularity. That is, we may import the original proof technique (sans flaws) from \cite{DDSS}.

\begin{corollary}\label{cor:finaluniqueness_main}
Let $A \in \cK(H_1)^d$ and $B \in \cK(H_2)^d$ be fully compressed tuples of compact operators satisfying $\cW(A) = \cW(B)$. Then $A$ and $B$ are unitarily equivalent.
\end{corollary}
\begin{proof}
Suppose that $\cW(A) = \cW(B)$. Then $A \mapsto B$ extends to a unital completely isometric isomorphism of $S_A$ onto $S_B$, and therefore it extends to a $*$-isomorphism of the corresponding C*-envelopes.
By Lemma \ref{lem:OGequivalence_fc_redundant}, $C^*(S_A) = C^*_e(S_A)$ and $C^*(S_B) = C^*_e(S_B)$.
We therefore have a $*$-isomorphism $\pi : C^*(S_A) \to C^*(S_B)$, which must be a direct sum of subrepresentations of the identity representation and perhaps the singular representation.

Let us first assume that $0$ is not a direct summand of either tuple.
Then $\pi$ is the sum of subrepresentations of the identity representation.
Each one of these representations appears at most once in the sum because $B$ is fully compressed and consequently multiplicity-free.
On the other hand, since $A$ does not have $0$ as a direct summand, each subrepresentation of the identity must appear at least once, since $\pi$ is injective.
We see that $\pi$ must be implemented by a unitary equivalence, as required.

Next, assume that $0$ is a direct summand of one of the tuples, say $A$. By Lemma \ref{lem:findimminfc}, $A$ then acts on a finite dimensional space, and it follows that $B$ also acts on a finite dimensional space. Arguing as above, we find that $\pi$ must be implemented by a unitary equivalence.
\end{proof}

Alternatively, recall that \cite[Proposition 6.7]{DDSSa} relates the same two conditions used above, multiplicity-free and trivial Shilov ideal, to minimality and nonsingularity. However, that result is not an equivalence. We may now extend both that result and Lemma \ref{lem:OGequivalence_fc_redundant} in the following theorem.

\begin{theorem}\label{thm:roundrobin_body}
Let $A \in \cK(H)^d$ be a tuple of compact operators. Then the following are equivalent.
\begin{enumerate}
\item\label{it:fullcommain} $A$ is fully compressed.
\item\label{it:mfShmain} $A$ is multiplicity-free, and the Shilov ideal of $S_A$ in $C^*(S_A)$ is trivial.
\item\label{it:msmain} $A$ is minimal and nonsingular.
\end{enumerate}
\end{theorem}
\begin{proof}
(\ref{it:fullcommain}) $\implies$ (\ref{it:msmain}): If $A$ is fully compressed, it is certainly minimal, so we need only show that $A$ is nonsingular. Theorem \ref{thm:minsingchar} shows that no singular compact tuple is fully compressed, as we may form a compression $A^\prime = \bigoplus\limits_{i=1}^N X^{(i)} \oplus Y^\prime$ with the same matrix range, where $Y^\prime$ is an infinite-dimensional compression of $Y$.

(\ref{it:msmain}) $\implies$ (\ref{it:mfShmain}): This is part of \cite[Proposition 6.7]{DDSSa}.

(\ref{it:mfShmain}) $\implies$ (\ref{it:fullcommain}): This is part of Lemma \ref{lem:OGequivalence_fc_redundant}.
\end{proof}
\begin{remark}
Corollary \ref{cor:finaluniqueness_main} can alternatively be deduced from Theorems \ref{thm:quotingminsingunique} and \ref{thm:roundrobin_body}.
\end{remark}

It follows from Theorem \ref{thm:roundrobin_body} that the tuples considered in \cite{Eve18} admit summands which are fully compressed.

\begin{example}\label{ex:Evertfcstuff}
Evert proves in \cite[Theorem 1.2]{Eve18} that if $d \geq 2$,  $T \in \cK(H)^d_{sa}$ has no finite-dimensional reducing subspaces, and $0$ is in the \lq\lq finite interior\rq\rq\hspace{0pt} of the noncommutative convex hull $K_T$, then $K_T$ is a closed matrix convex set which has no absolute extreme points. Since $K_T$ is also dense in $\cW(T)$ (see e.g. the explanation given in \cite[\S 1.3.1]{Eve18}), it immediately follows that $K_T = \cW(T)$ in this case. The assumption that $0$ is in the finite interior means that $0$ is a compression of $\bigoplus\limits_{1=i}^N T$ for some positive integer $N$, so from Proposition \ref{prop:zerocompression}, we have that $T$ is nonsingular. $T$ need not be minimal, but \cite[Corollary 6.8]{DDSSa} shows that $T$ admits a summand $T^\prime$ which is minimal and has $\cW(T^\prime) = \cW(T)$, and the proof produces $T^\prime$ which is also nonsingular. We conclude from Theorem \ref{thm:roundrobin_body} that $T^\prime$ is fully compressed.
\end{example}

Of course, the simplest examples which meet Evert's conditions are irreducible, and one may produce many such examples by modifying the coefficients used in \cite[Proposition 4.1]{Eve18}. Generally speaking, the assumption that two compact tuples be irreducible is very different from the assumption that they be fully compressed, so it is interesting that irreducible compact tuples are indeed always fully compressed (in particular, they are automatically minimal, and they cannot be singular as they do not fit into the mold of Theorem \ref{thm:minsingchar}). 

It should be noted that  irreducible compact operators are easily determined uniquely by their matrix ranges, as this was part of the study initiated by Arveson. Indeed, in \cite[Theorem 2.4.3]{Arv72}, Arveson showed that two irreducible GCR operators which have trivial Shilov ideals are unitarily equivalent if and only if their matrix ranges are equal. For compact operators, the Shilov ideal requirement is automatically satisfied (and the uniqueness result for compact operators is stated explicitly in \cite{Arv70}). From Arveson's theorem, and using the structure theorem for compact operators, one obtains a general classification theorem for compacts: {\em if $A = \bigoplus\limits_{i \in I} A_i$ and $B = \bigoplus\limits_{j \in J} B_j$ are two compact operators written as direct sums of irreducibles, then $A$ is unitarily equivalent to $B$ if and only if up to a bijection one has $I = J$ and $\cW(A_i) = \cW(B_i)$ for all $i \in I$}. 
A similar statement can be made regarding GCR operators that are decomposed into a direct integrals, but one needs to throw in the assumption that every constituent of the integral has trivial Shilov boundary. 
Recent classification theorems are somewhat different in spirit, as they use the matrix range of a tuple as a single entity.

At the bottom of page 304 in \cite{Arv72}, Arveson points out that the matrix range is not a complete invariant for irreducible operators in general, and he provides examples of non-unitarily equivalent, irreducible GCR operators having the same matrix range. In fact, he observes that if $S, T \in \cB(H)$ are contractions, both of which have spectrum that contains the unit circle, then $\cW(S) = \cW(T)$. By considering compressions of the unilateral shift to suitable subspaces \cite[p. 207]{Arv69}, one obtains such operators that are also irreducible (and even GCR). One can also readily see that all these examples satisfy $\cW(T) = \Wmin{}(\ol{\bD})$. In the following example, we present another large family of non unitarily-equivalent irreducible tuples with the same matrix ranges.

\begin{example}
Let $e_1, e_2, \ldots$ be the standard basis of $\ell^2(\bZ_+)$, and let $w_1, w_2, \ldots$ be another orthonormal basis satisfying the following technical conditions.
\begin{itemize}
\item For every $j$, $\langle w_1, e_j \rangle \not= 0$. $\begin{array}{c} \vspace{.2 cm} \end{array}$
\item $\lim\limits_{j \to \infty} \langle w_j, e_j \rangle = 1$.
\item The sum $R(i) := \sum\limits_{j=1}^\infty |\langle w_i, e_j \rangle|$ is finite for every $i$, and $\lim\limits_{i \to \infty} R(i) = 1$.
\item The sum $C(j) := \sum\limits_{i=1}^\infty |\langle w_i, e_j \rangle|$ is finite for every $j$, and $\lim\limits_{j\to\infty} C(j) = 1$.
\end{itemize}
Let $T = (T_1, \ldots, T_d)$ be such that $T_1$ is diagonal with respect to the basis $\{e_j\}_{j=1}^\infty$ and $T_2, \ldots, T_d$ are diagonal with respect to the basis $\{w_i\}_{i=1}^\infty$. We may select the eigenvalues in such a way that

\begin{itemize}
\item $T$ is irreducible,
\item $T_i \geq 0$ for each $i$,
\item $T_1 + \ldots + T_d \leq I$, and
\item $0$ and the standard basis vectors belong to $\cW_1(T)$ (but are not compressions of $T$).
\end{itemize}

It follows that $\cW_1(T)$ is precisely equal to the standard simplex $\Delta_d$ in $\bR^d$, and hence $\cW(T)$ is the unique matrix convex set $\Wmin{}(\Delta_d)$ over the simplex. There is enough freedom left in selecting the eigenvalues that we may produce uncountably many choices of $T$ which are not unitarily equivalent.
\end{example}

The matrix range considered above is also a set of the form $\Wmin{}(K)$, similar to Arveson's examples. It is not clear precisely which $K$ can be used in such constructions. Regardless, we note that none of the examples considered are fully compressed, and we are led to the following questions.

\begin{question}
For tuples $A, B \in \cB(H)^d$ which are not necessarily compact, if $A$ and $B$ are fully compressed and $\cW(A) = \cW(B)$, does it follow that $A$ is unitarily equivalent to $B$?
\end{question}

\begin{question}\label{ques:fctype}
Let $\cS$ be a closed and bounded matrix convex set. When can $\cS$ be written as the matrix range of a fully compressed tuple $T \in \cB(H)^d$?
\end{question}

One may immediately apply \cite[Theorem 3.26]{Passer} to see that fully compressed \textit{normal} tuples, which are necessarily also minimal, are uniquely determined by their matrix ranges. Moreover, normal tuples are minimal if and only if they are fully compressed.

\begin{theorem}\label{thm:cmdm_normal}
Let $N \in \cB(H)^d$ be a normal tuple.
Then $N$ is fully compressed if and only if it is minimal. 
In particular, this occurs if and only if there is a compact convex set $K$ such that the set of isolated extreme points $\cI_K$ has $\overline{\emph{ext}(K)} = \overline{\cI_K}$ and $N$ is unitarily equivalent to $\bigoplus\limits_{\lambda \in \cI_K} \lambda$.
\end{theorem}
\begin{proof}
We need only prove that minimal normal tuples are fully compressed. If $N \in \cB(H)^d$ is normal and minimal for its matrix range, then from \cite[Theorem 3.26]{Passer}, we have that $N$ is a multiplicity-free diagonal operator, the eigenvalues of $N$ lie at the set $\cI_K$ of isolated extreme points of some compact convex set $K$, and $K$ satisfies $\ol{\cI_K} = \ol{\operatorname{ext}(K)}$. 
Given a joint eigenvalue $\lambda \in \cI_K$, let $v_\lambda$ denote a corresponding eigenvector.

Let $G \subseteq H$ be a closed subspace with $\cW(P_G N|_G) = \cW(N)$, so in particular $K = \cW_1(P_G N|_G)$. 
Fix $\lambda \in \cI_K$, and fix an $\bR$-affine transformation $\Phi: \bC^d \to \bR$ and a number $0 < r < 1$ such that $\Phi(\lambda) = 1$ but $-r \leq \Phi(\gamma) \leq r$ for any $\gamma \in \cI_K \setminus \{\lambda\}$. 
Next, let $A := \Phi(N) \in \cB(H)_{sa}$, so that $A$ is a diagonal self-adjoint operator. In particular, $A$ has a joint eigenvalue $1$ at $v_\lambda$ and eigenvalues of magnitude at most $r$ at $v_\gamma$, $\gamma \not= \lambda$.

Since $\lambda \in \cW_1(P_G N|_G)$, we have that $1 \in \cW_1(P_G A|_G)$. 
Keeping in mind that $P_G A|_G$ is a self-adjoint contraction, it follows that we may find unit vectors $g_n \in G$ such that $\langle A g_n, g_n \rangle \to 1$. 
However, if one applies the decomposition $\cH = \bC v_\lambda \oplus (\bC v_\lambda)^\perp$ (where both subspaces are reducing for $A$) to obtain $g_n = b_n + c_n$, it follows that
\bes
\langle A g_n, g_n \rangle = \langle A b_n, b_n \rangle + \langle A c_n, c_n \rangle \leq ||b_n||^2 + r ||c_n||^2.
\ees
Since $|r| < 1$ and $||b_n||^2 + ||c_n||^2 = 1$, we must have that $||b_n|| \to 1$ and $||c_n|| \to 0$. 
That is, after a unimodular rescaling, $g_n$ converges to $v_\lambda$, and hence $v_\lambda \in G$. This applies to each eigenvector $v_\lambda$, so $G = H$, and finally, $N$ is fully compressed.
\end{proof}

A fully compressed normal tuple is the direct sum of the isolated extreme points of the first level of its matrix range $\Wmin{}(K)$, and in particular, these points are the crucial matrix extreme points of $\Wmin{}(K)$. In general, however, a fully compressed tuple need not be the direct sum of crucial matrix extreme points, as seen in Example \ref{ex:Evertfcstuff}. That is, the boundary representations of the corresponding operator system might be exclusively infinite-dimensional.

\vskip 5pt

\noindent{\bf Acknowledgement.} We thank the anonymous referee for insightful comments.

%%%%%%%%%%%%%%%%%%%%%%%%%%%%%%%%%%%%%%%%%%%%%%%%%%%%%%%%%%%%%%%%%%%%%%%%%%%%
%%%%%%%%%%%%%%%%%%%%%%%%%%%%%%%%%%%%%%%%%%
\bibliographystyle{amsplain}

\end{document}